\documentclass[review]{elsarticle}
\usepackage[pdftex,nomarginpar,width=130mm,height=220mm]{geometry}
\usepackage{lineno}
\usepackage{amsfonts,amssymb,amsmath,amsthm}
\newtheorem{theorem}{Theorem}[section]
\newtheorem{lemma}{Lemma}[section]
\newtheorem{pro}{Proposition}[section]

\theoremstyle{remark}

\theoremstyle{remark}

\newtheorem{example}{Example}[section]

\DeclareMathOperator{\Span}{Span}
\DeclareMathOperator{\tr}{trace}
\DeclareMathOperator{\Ric}{Ric}

\newcommand{\co}{\nabla}

\newcommand{\di}{\mathcal{D}}

\modulolinenumbers[5]

\journal{Journal of \LaTeX\ Templates}

\begin{document}

\begin{frontmatter}

\title{On contact pseudo-metric manifolds satisfying a nullity condition}

\author[mymainaddress]{Narges Ghaffarzadeh}
\ead{n.ghaffarzadeh@tabrizu.ac.ir}
\author[mymainaddress, mycorrespondingauthor]{Morteza Faghfouri}
\cortext[mycorrespondingauthor]{Corresponding author}
\ead{faghfouri@tabrizu.ac.ir}
\address[mymainaddress]{Department of Pure Mathematics, Faculty of Mathematical Sciences, University of Tabriz, Tabriz, Iran.}

\begin{abstract}
In this paper, we aim to introduce and study $(\kappa, \mu)$-contact pseudo-metric manifold and  prove that if the $\varphi$-sectional curvature of any point of $M$ is independent of the choice of $\varphi$-section at the point, then it is constant on $M$ and  accordingly  the curvature tensor. Also, we introduce generalized  $(\kappa, \mu)$-contact pseudo-metric manifold and  prove for $n>1$, that a non-Sasakian generalized  $(\kappa, \mu)$-contact pseudo-metric manifold is a $(\kappa, \mu)$-contact pseudo-metric manifold.
\end{abstract}


\end{frontmatter}

\linenumbers
 \section{Introduction}
 Contact pseudo-metric structures  were first introduced by Takahashi \cite{Takahashi:SasakianManifoldWithPseud}. He defined Sasakian manifold with pseudo-metric and the classification of Sasakian manifolds of constant $\phi$-sectional curvatures. Next, K. L. Duggal  \cite{Duggal:SpaceTimeManifoldsContactStructures} and A. Bejancu \cite{Bejancu1993}  studied contact pseudo-metric structures as  a generalization of contact Lorentzian  structures and contact Riemannian structures. Recently, contact pseudo-metric manifolds and curvature of $K$-contact pseudo-Riemannian manifolds  have been studied by Calvaruso and Perrone \cite{Calvaruso.Perrone:ContactPseudoMtricManifolds} and Perrone \cite{Perrone:CurvatureKcontact}, respectively. Also Perrone  \cite{Perrone2014}, Perrone  investigated  contact pseudo-metric manifolds of constant curvature and CR manifolds.

In \cite{Blair:Contactmetricmanifoldssatisfyingnullitycondition}, D. E. Blair et al.  introduced $(\kappa,\mu)$-contact Riemannian manifold. Since then, many researchers have studied the structure  \cite{koufogiorgos:Contact.constant.curvature,boeckx:a.full.classification.of.contact.metric,Koufogiorgos:OnTheExistenceContactMetricManifolds,
faghfouriGhaffarzade:doublywarped,faghfouriGhaffarzade:invariant,faghfouriGhaffarzade:CtotallyReal}.

In this paper, we introduce and study $(\kappa, \mu)$-contact pseudo-metric manifold. The paper is organized as follows. Section $2$ contains some necessary background on contact pseudo-metric manifolds. After introducing $(\kappa,\mu)$-contact pseudo-metric manifold in section 3, we prove some relationships. In this section, we also prove if the $\varphi$-sectional curvature of any point of $M$ is independent of the choice of $\varphi$-section at the point, then it is constant on $M$ and we find the curvature tensor. In fact, our main purpose in this paper is to find  the curvature tensor of $(\kappa,\mu)$-contact pseudo-metric manifolds. In addition, we show that $M$ has constant $\varphi$-sectional curvature if and only if $\mu=\varepsilon\kappa+1$ when $\kappa\neq\varepsilon$. In section $4$, we introduce generalized  $(\kappa, \mu)$-contact pseudo-metric manifold. In this section, we also  prove for $n>1$, that a non-Sasakian generalized  $(\kappa, \mu)$-contact pseudo-metric manifold is a $(\kappa, \mu)$-contact pseudo-metric manifold.

\section{Preliminaries}
A $(2n+1)$-dimensional differentiable manifold $M$ is called an almost contact pseudo-metric manifold if there is an almost contact pseudo-metric structure $(\varphi,\xi,\eta,g)$ consisting of a $(1, 1)$ tensor field $\varphi$, a vector field $\xi$, a $1$-form $\eta$ and a compatible pseudo-Riemannian metric $g$ satisfying
\begin{gather}
\eta (\xi)= 1,\varphi^2(X)=-X+\eta(X)\xi, \label{001}\\
g(\varphi X,\varphi Y)=g(X,Y)-\varepsilon\eta(X)\eta(Y),\label{002}
\end{gather}
where $\varepsilon=\pm1$ and  $X,Y\in\Gamma(TM)$.
Remark that, by (\ref{001}) and (\ref{002}), we have
\begin{gather}
\varphi\xi =0,\eta\circ\varphi=0,\\
\eta(X) =\varepsilon g(\xi, X),\\
g(\varphi X,Y)=-g(X,\varphi Y),
\end{gather}
and $\varphi$ has rank $2n$.
 In particular, $g(\xi,\xi)=\varepsilon$ and so, the characteristic vector field $\xi$ is either space-like or time-like, but cannot be light-like and the signature of an associated metric is either $(2p+1,2n-2p)$ or $(2p,2n-2p-1)$.
An almost contact pseudo-metric structure becomes a contact pseudo-metric structure if $d\eta=\Phi$, where $\Phi(X,Y)=g( X,\varphi Y)$ is the fundamental $2$-form of $M$.

An almost contact pseudo-metric structure of $M$ is called a normal structure if  $[\varphi,\varphi]+2d\eta\otimes\xi=0$. A normal
contact pseudo-metric structure is called a Sasakian structure. It can be proved that an almost contact pseudo-metric manifold is Sasakian iff
\begin{align}
(\co_X\varphi)Y=g(X,Y)\xi-\varepsilon\eta(Y)X,
\end{align}
for any $X, Y\in \Gamma(TM)$ or  equivalently, a contact pseudo-metric structure is a Sasakian structure iff $R$ satisfies
\begin{align}\label{eq:sasakicurvature}
R(X,Y)\xi=\eta(Y)X-\eta(X)Y,
\end{align}
for $X,Y\in\Gamma(TM)$, where $R(X,Y)=[\co_X,\co_Y]-\co_{[X,Y]}$ is the curvature tensor and $\co$ is the Levi-Civita connection\cite{kumar.rani.nagachi:Onsectionalcurvatures,Calvaruso.Perrone:ContactPseudoMtricManifolds}.
In a contact pseudo-metric manifold $M^{2n+1}(\varphi,\xi,\eta,g)$, we define the $(1, 1)$-tensor fields $\ell$ and $h$ by
\begin{align}\label{050}
\ell X=R(X,\xi)\xi,\quad hX=\frac{1}{2}(\mathcal{L}_{\xi}\varphi)(X),
\end{align}
where $\mathcal{L}$ denotes the Lie derivative.
The tensors $h$ and $\ell$ are self-adjoint operators satisfying(\cite{Calvaruso.Perrone:ContactPseudoMtricManifolds,Perrone:CurvatureKcontact})
\begin{gather}
\tr(h)=\tr( h\varphi)=0,\label{051}\\
\eta\circ h=0, \quad\ell\xi =0,\label{055}\\
h\varphi=-\varphi h,\label{013}\\
h\xi =0,\label{052}\\
\co_{X}\xi =-\varepsilon\varphi X-\varphi hX,\label{033}\\
(\co_{X}\varphi)Y=\varepsilon g(\varepsilon X+hX,Y)\xi-\eta(Y)(\varepsilon X+hX),\label{031}\\
\co_{\xi}\varphi =0.\label{072}
\end{gather}
Due to the relation of (\ref{013}), if $X$ is an eigenvector of $h$ corresponding to the eigenvalue $\lambda$, then $\varphi X$ is also an eigenvector of $h$ corresponding to the eigenvalue $-\lambda$.

Let $M^{2n+1}(\varphi,\xi,\eta,g)$ be a contact pseudo-metric manifold and $X\in\ker\eta$, either space-like or time-like. We put
\begin{gather}
K(X,\xi)=\dfrac{\mathcal{R}(X,\xi ,X,\xi)}{\varepsilon g(X,X)}=\dfrac{g(\ell X,X)}{\varepsilon g(X,X)},\label{060}\\
K(X,\varphi X)=\dfrac{\mathcal{R}(X,\varphi X,X,\varphi X)}{g(X,X)^{2}}.
\end{gather}
We call $K(X,\xi)$ the $\xi$-sectional curvature determined by $X$, and $K(X,\varphi X)$ the $\varphi$-sectional curvature determined by $X$, where $\mathcal{R}(X,Y,Z,W)=g(R(Z,W)Y,X)$. A Sasakian manifold with constant $\varphi$-sectional curvature $c$
is called a Sasakian space form and is denoted by
$M(c).$
\begin{lemma}[\cite{Calvaruso.Perrone:ContactPseudoMtricManifolds}]
Let $M^{2n+1}(\varphi,\xi,\eta,g)$ be a contact pseudo-metric manifold. Then:
\begin{gather}
\co_{\xi}h=\varphi-\varphi\ell-\varphi h^{2},\label{073}\\
\varphi\ell\varphi -\ell =2(\varphi^{2}+h^{2}),\label{027}\\
\Ric(\xi,\xi)=2n-\tr(h^{2}),\label{035}\\
\mathcal{R}(\xi ,X,Y,Z)=\varepsilon(\co _{X}\Phi)(Y,Z)+g((\co_{Y}\varphi h)Z,X)-g((\co_{Z}\varphi h)Y,X).\label{032}
\end{gather}
\end{lemma}
\section{$(\kappa, \mu)$-contact pseudo-metric manifold}
 A $(\kappa ,\mu)$-nullity distribution of a contact pseudo-metric manifold  $M^{2n+1}(\varphi ,\xi,\eta, g)$  is a distribution
\begin{align}\label{61}\begin{split}
N_{p}(\kappa ,\mu)=\lbrace Z\in T_{p}M:R(X,Y)Z=&\kappa(g(Y,Z)X-g(X,Z)Y)\\
&+\mu(g(Y,Z)hX-g(X,Z)hY)\rbrace,\end{split}
\end{align}
where $(\kappa ,\mu)\in \mathbb{R}^2$.
Thus, the characteristic vector field $\xi$ belongs to the $(\kappa ,\mu)$-distribution iff
\begin{align}\label{62}
R(X,Y)\xi= \varepsilon\kappa(\eta(Y )X -\eta(X)Y ) +\varepsilon \mu(\eta(Y )hX -\eta(X)hY ).
\end{align}
 If a contact pseudo-metric manifold satisfying (\ref{62}), we call  $(\kappa, \mu)$-contact pseudo-metric manifold. The class of $(\kappa, \mu)$-contact pseudo-metric manifold contains the class of Sasakian manifolds, which we get for $\kappa=\varepsilon$ (and hence $h=0$, by (\ref{eq:sasakicurvature})).
\begin{lemma}\label{071}
Let $M^{2n+1}(\varphi,\xi,\eta,g)$ be a  $(\kappa, \mu)$-contact pseudo-metric manifold. Then, we have
\begin{gather}
\ell\varphi-\varphi\ell=2\varepsilon\mu h\varphi,\label{053}\\
 h^2=(\varepsilon\kappa-1)\varphi ^2,\quad \varepsilon\kappa \leq 1, \quad \text{ and }\quad \kappa=\varepsilon  \text{ iff }  M^{2n+1} \text{ is  Sasakian }, \label{030}\\
R(\xi ,X)Y=\kappa(g(X,Y)\xi-\varepsilon\eta(Y)X)+\mu(g(hX,Y)\xi -\varepsilon\eta(Y)hX),\label{023}\\
Q\xi=2n\kappa\xi,\text{\quad $Q$ is the Ricci operator},\label{054}\\
(\co_{X}h)Y-(\co _{Y}h)X=(1-\varepsilon\kappa)\{2\varepsilon g(X,\varphi Y)\xi+\eta(X)\varphi Y-\eta(Y)\varphi X\}\nonumber\\
\quad \quad \quad \quad\quad \quad \quad \quad\quad\quad+\varepsilon(1-\mu)\{\eta(X)\varphi hY-\eta(Y)\varphi hX\},\label{48}\\
\xi\kappa =0,\label{070}
\end{gather}
where  $X,Y\in\Gamma(TM)$.
\end{lemma}
\begin{proof}
Using (\ref{052}), we obtain
$$\ell X=\varepsilon\kappa(X-\eta(X)\xi)+\varepsilon\mu hX,$$
for $X\in\Gamma(TM)$. Replacing $X$ by $\varphi X$ and at the same time applying $\varphi$, we obtain
\begin{align}\label{026}
\ell\varphi=\varepsilon\{\kappa\varphi+\mu h\varphi\}\quad\text{and}\quad\varphi\ell=\varepsilon\{\kappa\varphi+\mu\varphi h\}.
\end{align}
Subtracting (\ref{026}) and using (\ref{013}), we get (\ref{053}).\\
By using the relations (\ref{027}), (\ref{013}), (\ref{026}), (\ref{052}) and (\ref{001}), we deduce the first part of (\ref{030}). Now since $h$ is symmetric, from the second part of (\ref{001}), we have $\varepsilon\kappa\leq1$. Moreover, $\kappa=\varepsilon$ iff $h=0$. Using (\ref{62}) and (\ref{eq:sasakicurvature}), the proof of (\ref{030}) is completed.

Using (\ref{62}), we get (\ref{023}) and $g(R(\xi,X)Y,Z)=g(R(Y,Z)\xi,X)$.\\
For the relation of  (\ref{054}), let  $\{E_{1},\ldots, E_{n},E_{n+1}=\varphi E_{1}, \ldots,E_{2n}=\varphi E_{n},E_{2n+1}=\xi\}$  be a (local)$\varphi$-basis  of  $M$. For any index
$i =1,\ldots,2n$, $\{\xi, E_{i}\}$ spans a non-degenerate plane on the tangent space at each point where the basis is defined. Then the definition of the Ricci operator $Q$, (\ref{023}), (\ref{051}) and (\ref{052}) give
\begin{align*}
\Ric (\xi,X)=&\sum_{i=1}^{2n+1}\varepsilon_{i}g(R(E_{i},\xi)X,E_{i})\\
=&\sum_{i=1}^{2n+1}\varepsilon_{i}\{\kappa[\varepsilon\eta(X)g(E_{i},E_{i})-\varepsilon g(E_{i},X)\eta(E_{i})]\\
&+\mu[\varepsilon\eta(X)g(hE_{i},E_{i})-\varepsilon g(hE_{i},X)\eta(E_{i})]\}\\
=&\varepsilon\kappa\eta(X)\sum_{i=1}^{2n+1}\varepsilon_{i}^{2}-\varepsilon\kappa\eta(X)+\mu\varepsilon\eta(X)\tr(h)\\
=&\varepsilon\kappa\eta(X)(2n+1)-\varepsilon\kappa\eta(X)=(2n+1-1)\varepsilon\kappa\eta(X)=2n\varepsilon\kappa\eta(X),
\end{align*}
so, we have
(\ref{054}). Now with using (\ref{031}) and the symmetry of $h$, we get
\begin{align*}
(\co_{X}\varphi h)Y-(\co _{Y}\varphi h)X=\varphi((\co_{X}h)Y-(\co _{Y}h)X),
\end{align*}
for any vector fields $X,Y$ on $M$ and hence (\ref{032}) is reduced to
$$R(Y,X)\xi=\eta(X)(Y +\varepsilon hY) - \eta(Y)(X+\varepsilon hX) + \varphi((\co_{X}h)Y - (\co_{Y}h)X).$$
Comparing this equation with (\ref{62}), we have
\begin{align}\label{056}\begin{split}
\varphi((\co_{X}h)Y-(\co_{Y}h)X)=&(\varepsilon\kappa-1)(\eta(X)Y-\eta(Y)X)\\
&+\varepsilon(\mu-1)(\eta(X)hY-\eta(Y)hX).\end{split}
\end{align}
Using (\ref{033}), the symmetry of $h$ and $\co_{X}h$, we obtain
\begin{align}\label{057}
g((\co_{X}h)Y-(\co_{Y}h)X,\xi)=2(\varepsilon\kappa-1)g(Y,\varphi X).
\end{align}
Acting now by $\varphi$ on (\ref{056}) and using (\ref{057}), we get \eqref{48}.
\end{proof}
\begin{lemma}
Let $M^{2n+1}(\varphi,\xi,\eta,g)$ be a $(\kappa, \mu)$-contact pseudo-metric manifold. Then for all  $X,Y,Z\in\Gamma(TM)$, we have
\begin{align}\label{006}\begin{split}
R(X,Y)\varphi Z=&\varphi R(X,Y)Z+\{(1-\varepsilon\kappa)[\eta(X)g(\varphi Y,Z)-\eta(Y)g(\varphi X,Z)]\\
&+\varepsilon(1-\mu)[\eta(X)g(\varphi hY,Z)-\eta(Y)g(\varphi hX,Z)]\}\xi\\
&-g(Y+\varepsilon hY,Z)(\varepsilon \varphi X+\varphi hX)+g(X+\varepsilon hX,Z)(\varepsilon \varphi Y+\varphi hY)\\
&-g(\varepsilon\varphi Y+\varphi hY,Z)(X+\varepsilon hX)+g(\varepsilon\varphi X+\varphi hX,Z)(Y+\varepsilon hY)\\
&-\eta(Z)\{(1-\varepsilon\kappa)[\eta(X)\varphi Y-\eta(Y)\varphi X]+\varepsilon(1-\mu)[\eta(X)\varphi hY-\eta(Y)\varphi hX]\}.
\end{split}
\end{align}
\end{lemma}
\begin{proof}
Assume that $p\in M$ and $X, Y, Z$ local vector fields on a neighborhood of $p$, such that $$(\co X)_{p}=(\co Y)_{p}=(\co Z)_{p}=0.$$ The Ricci identity for $\varphi$:
\begin{align}
R( X, Y)\varphi Z-\varphi R( X, Y)Z=(\co_{X}\co_{Y}\varphi)Z-\co_{Y}(\co_{X}\varphi)Z-(\co_{[X,Y]}\varphi)Z,
\end{align}
at the point $p$, takes the form
\begin{align}\label{034}
R(X,Y)\varphi Z-\varphi R(X,Y)Z=\co_{X}(\co_{Y}\varphi)Z-\co_{Y}(\co_{X}\varphi)Z.
\end{align}
On the other hand, combining (\ref{033}) and (\ref{031}), we have at $p$
\begin{align}\label{058}\begin{split}
\co_{X}(\co_{Y}\varphi)Z-\co_{Y}(\co_{X}\varphi)Z=&\varepsilon g((\co_{X}h)Y-(\co_{Y}h)X,Z)\xi-g(Y+\varepsilon hY,Z)(\varepsilon\varphi X+\varphi hX)\\
&+\varepsilon g(\varepsilon\varphi X+\varphi hX,Z)(\varepsilon Y+hY)+g(X+\varepsilon hX,Z)(\varepsilon\varphi Y+\varphi hY)\\
&-\varepsilon g(Z,\varepsilon \varphi Y+\varphi hY)(\varepsilon X+hX)-\eta(Z)((\co_{X}h)Y-(\co_{Y}h)X)
\end{split}
\end{align}
Now equation (\ref{006}) is a straightforward combination of the (\ref{058}), (\ref{034}) and (\ref{48}).
\end{proof}
\begin{theorem}\label{037}
Let $M^{2n+1}(\varphi,\xi,\eta,g)$ be a $(\kappa, \mu)$-contact pseudo-metric manifold.  Then $\varepsilon\kappa\leq 1$. If $\kappa= \varepsilon$, then $h=0$ and $M^{2n+1}$ is a  Sasakian-space-form and if $\varepsilon\kappa<1$, then $M^{2n+1}$ admits three mutually orthogonal and integrable distributions $\di(0)=\Span\{\xi\}, \di(\lambda)$ and $\di(-\lambda)$, defined by the eigenspaces of $h$, where $\lambda=\sqrt{1-\varepsilon\kappa}$.
\end{theorem}
\begin{proof}
By $\xi\in N(\kappa,\mu)$, we can verify $\Ric(\xi,\xi)=2n\varepsilon\kappa$. Then, (\ref{035}) implies
$\varepsilon\kappa\leq1$. Now, we suppose $\varepsilon\kappa<1$. Then since $h$ is symmetric, the relations (\ref{052}) and (\ref{031}) imply that the restriction $h\vert\di$ of $h$ to the contact distribution $\di$ has eigenvalues $\lambda=\sqrt{1-\varepsilon\kappa}$ and $-\lambda$. By $\di(\lambda)$ and  $\di(-\lambda)$, we denote the distributions defined by the eigenspaces of $h$ corresponding to $\lambda$ and $-\lambda$, respectively. By  $\di(0)$, we denote the distribution defined by $\xi$. Then these three distributions are mutually orthogonal. Let  $X\in\di(\lambda)$, Then $hX=\lambda X$ and the relation of (\ref{013}) imply
$h(\varphi X)=-\lambda(\varphi X)$. Hence, we have $\varphi X\in\di(-\lambda)$. This means that the dimension of  $\di(\lambda)$ and $\di(-\lambda)$ are equal to $n$. We prove that  $\di(\lambda)$ ( $\di(-\lambda)$, resp.) is integrable. Let $X,Y\in\di(\lambda)$ ($\di(-\lambda)$, resp.). Then
$$\co_{X}\xi=-\varepsilon\varphi X-\varphi hX=-(\varepsilon\pm\lambda)\varphi X,$$
and $\co_{Y}\xi=-(\varepsilon\pm\lambda)\varphi Y$. So, $g(\co_{X}\xi,Y)=g(\co_{Y}\xi,X)$ holds. Thus, $d\eta(X,Y)=0$ and $\eta([X,Y])=0$ follow. $X,Y\in\di(\lambda)$ and $\xi\in N(\kappa,\mu)$ imply $R(X,Y)\xi=0$. On the other hand,
\begin{align}\label{059}\begin{split}
0&=\co_{X}\co_{Y}\xi-\co_{Y}\co_{X}\xi-\co_{[X,Y]}\xi\\
&=-(\varepsilon\pm\lambda)\co_{X}(\varphi Y)+(\varepsilon\pm\lambda)\co_{Y}(\varphi X)+\varepsilon\varphi([X,Y])+\varphi h([X,Y])\\
&=-(\varepsilon\pm\lambda)\{(\co_{X}\varphi)Y-(\co_{Y}\varphi)X\}\mp\lambda\varphi([X,Y])+\varphi h([X,Y]).
\end{split}
\end{align}
By (\ref{031}), the first term of the last line (\ref{059})
vanishes. And so, we obtain
$$\varphi h([X,Y])=\mp\lambda\varphi([X,Y]),$$
which together with $\eta([X,Y])=0$ implies $[X,Y]\in\di(\lambda)$ ($\mathcal{D}(-\lambda)$, resp.).
\end{proof}
\begin{pro}\label{47}
Let $M^{2n+1}(\varphi,\xi,\eta,g)$ be a $(\kappa, \mu)$-contact pseudo-metric manifold with $\varepsilon\kappa <1$,
then
\begin{itemize}
  \item If $X,Y\in \di(\lambda)$ (resp. $\di(-\lambda)$),  then $\co_{X}Y\in \di(\lambda)$ (resp. $\di(-\lambda)$).
  \item If $X\in \di(\lambda), Y\in \di(-\lambda)$,  then $\co_{X}Y (resp. \co_{Y}X )$  has no component in $\di(\lambda)$ (resp. $\di(-\lambda)$).
\end{itemize}
\end{pro}

 \begin{lemma}\label{042}
Let $M^{2n+1}(\varphi,\xi,\eta,g)$ be a $(\kappa, \mu)$-contact pseudo-metric manifold. Then for any vector fields $X,Y$ on $M$, we have
\begin{align}\label{041}\begin{split}
(\co_{X}h)Y=&\{(\varepsilon-\kappa)g(X,\varphi Y)+g(X,h\varphi Y)\}\xi\\
&+\eta(Y)[h(\varepsilon\varphi X+\varphi hX)]-\varepsilon\mu\eta(X)\varphi hY.
\end{split}
\end{align}
\end{lemma}
\begin{proof}
Let $\varepsilon\kappa< 1$ and $X,Y\in\di(\lambda)$(resp., $\di(-\lambda)$). Then from Proposition \ref{47}, we have $\co_{X}Y\in\di(\lambda)$ (resp., $\di(-\lambda)$). Then one easily proves that
\begin{align}\label{038}
(\co_{X}h)Y=0.
\end{align}
Suppose now that $X\in\di(\lambda)$, $Y\in\di(-\lambda)$ and $\{E_{i},\varphi E_{i},\xi\}$ be a (local) $\varphi$-basis of vector fields on $M$ with $E_{i}\in\di(\lambda)$ and so $\varphi E_{i}\in\di(-\lambda)$. For any index $i =1,\cdots,2n$, $\{\xi, E_{i}\}$ spans a non-degenerate plane on the tangent space at each point where the basis is defined. Then using Proposition \ref{47} and  the relations (\ref{052}), (\ref{001}) and (\ref{033}), we calculate
\begin{align*}
h\co_{X}Y&=h\{\sum_{i=1}^{n}\varepsilon_{i}g(\co_{X}Y,\varphi E_{i})\varphi E_{i}+\varepsilon g(\co_{X}Y,\xi)\xi\}\\
&=\sum_{i=1}^{n}\varepsilon_{i}g(\co_{X}Y,\varphi E_{i})h\varphi E_{i}\\
&=\lambda\varphi\sum_{i=1}^{n}\varepsilon_{i}g(\varphi\co_{X}Y,E_{i})E_{i}\\
&=\lambda\varphi^{2}(\co_{X}Y)\\
&=\lambda(-\co_{X}Y+\varepsilon g(\co_{X}Y,\xi)\xi)\\
&=\lambda(-\co_{X}Y-\varepsilon g(Y,\co_{X}\xi)\xi)\\
&=\lambda(-\co_{X}Y-\varepsilon g(Y,-\varepsilon\varphi X-\varphi hX)\xi)\\
&=\lambda(-\co_{X}Y+g(Y,\varphi X+\varepsilon\varphi hX)\xi)\\
&=\lambda(-\co_{X}Y-g(\varphi Y,X+\varepsilon hX)\xi)\\
&=\co_{X}hY-\lambda(1+\varepsilon\lambda)g(X,\varphi Y)\xi,
\end{align*}
and so
\begin{align}\label{039}
(\co_{X}h)Y=\lambda(1+\varepsilon\lambda)g(X,\varphi Y)\xi,
\end{align}
Similarly, we obtain
\begin{align}\label{040}
(\co_{Y}h)X=\lambda(\varepsilon\lambda -1)g(Y,\varphi X)\xi,
\end{align}
Suppose now that $X,Y$ are arbitrary vector fields on $M$ and write
$$X=X_{\lambda}+X_{-\lambda}+\eta(X)\xi,$$
and
$$Y=Y_{\lambda}+Y_{-\lambda}+\eta(Y)\xi,$$
where $X_{\lambda}$ (resp., $X_{-\lambda}$) is the component of $X$ in $\di(\lambda)$ (resp., $\di(-\lambda)$). Then using (\ref{038}), (\ref{039}), (\ref{040}) and $\co_{\xi}h=\varepsilon\mu h\varphi$, which follows from (\ref{48}), we get by a direct computation
\begin{align}\label{063}
\begin{split}
(\co_{X}h)Y=&\varepsilon\lambda^{2}\{g(X_{\lambda},\varphi Y_{-\lambda})+g(X_{-\lambda},\varphi Y_{\lambda})\}\xi\\
&+\lambda\{g(X_{\lambda},\varphi Y_{-\lambda})-g(X_{-\lambda},\varphi Y_{\lambda})\}\xi\\
&+\eta(Y)h(\varepsilon\varphi X+\varphi hX)-\varepsilon\mu\eta(X)\varphi hY.
\end{split}
\end{align}
On the other hand, we easily find that
\begin{gather}
g(hX,\varphi Y)=\lambda\{g(X_{\lambda},\varphi Y_{-\lambda})-g(X_{-\lambda},\varphi Y_{\lambda})\},\label{064}\\
g(hX,h\varphi Y)=\lambda^{2}\{g(X_{-\lambda},\varphi Y_{\lambda})+g(X_{\lambda},\varphi Y_{-\lambda})\}.\label{065}
\end{gather}
The relations (\ref{064}) and (\ref{065}) with (\ref{063}), give the required equation (\ref{041}). Note that for $\kappa=\varepsilon$ (and so $h=0$), (\ref{041}) is valid identically and the proof is completed.
\end{proof}
\begin{lemma}
Let $M^{2n+1}(\varphi,\xi,\eta,g)$ be a $(\kappa, \mu)$-contact pseudo-metric manifold. Then for any vector fields $X,Y,Z$ on $M$. We have
\begin{align}\label{045}\begin{split}
R(X,Y)hZ-hR(X,Y)Z=&\{\kappa[\eta(X)g(hY,Z)-\eta(Y)g(hX,Z)]\\
&+\mu(\varepsilon\kappa-1)[\eta(Y)g(X,Z)-\eta(X)g(Y,Z)]\}\xi\\
&+\kappa\{g(Y,\varphi Z)\varphi hX-g(X,\varphi Z)\varphi hY+g(Z,\varphi hY)\varphi X\\
&-g(Z,\varphi hX)\varphi Y+\varepsilon\eta(Z)[\eta(X)hY-\eta(Y)hX]\}\\
&-\mu\{\eta(Y)[(\varepsilon-\kappa)\eta(Z)X+\mu\eta(X)hZ]\\
&-\eta(X)[(\varepsilon-\kappa)\eta(Z)Y+\mu\eta(Y)hZ]+2\varepsilon g(X,\varphi Y)\varphi hZ\}.
\end{split}
\end{align}
\end{lemma}
\begin{proof}
The Ricci identity for $h$ is
\begin{align}\label{043}
R(X,Y)hZ-hR(X,Y)Z=(\co_{X}\co_{Y}h)Z-(\co_{Y}\co_{X}h)Z-(\co_{[X,Y]}h)Z.
\end{align}
Using Lemma \ref{042}, the relations (\ref{030}), (\ref{013}) and the fact that $\co_{X}\varphi$ is antisymmetric, we obtain
\begin{align*}
(\co_{X}\co_{Y}h)Z=&\{(\varepsilon-\kappa)[g(\co_{X}Y,\varphi Z)-g((\co_{X}\varphi)Y,Z)]\\
&+g(\co_{X}Y,h\varphi Z)+g(\co_{X}(h\varphi)Y,Z)\}\xi\\
&+\{(\varepsilon-\kappa)g(Y,\varphi Z)+g(Y,h\varphi Z)\}\co_{X}\xi\\
&+\varepsilon g(Z,\co_{X}\xi)[\varepsilon h\varphi Y+(\varepsilon\kappa-1)\varphi Y]\\
&+\eta(Z)\{\varepsilon[(\co_{X}h\varphi)Y+h\varphi(\co_{X}Y)]+(\varepsilon\kappa-1)[(\co_{X}\varphi)Y+\varphi(\co_{X}Y)]\}\\
&-\varepsilon\mu\{\eta(\co_{X}Y)\varphi hZ+\varepsilon g(Y,\co_{X}\xi)\varphi hZ+\eta(Y)(\co_{X}\varphi h)Z\}.
\end{align*}
So, using also (\ref{041}), (\ref{033}), (\ref{031}) and
Lemma \ref{042}, equation (\ref{043}) yields
\begin{align}\label{044}\begin{split}
R(X,Y)hZ-h&R(X,Y)Z\\
=&\{(\kappa-\varepsilon)g((\co_{X}\varphi)Y-(\co_{Y}\varphi)X,Z)+g((\co_{X}h\varphi)Y-(\co_{Y}h\varphi)X,Z)\}\xi\\
&+\{(\varepsilon-\kappa)g(Y,\varphi Z)+g(Y,h\varphi Z)\}\co_{X}\xi\\
&-\{(\varepsilon-\kappa)g(X,\varphi Z)+g(X,h\varphi Z)\}\co_{Y}\xi\\
&+g(Z,\co_{X}\xi)[h\varphi Y+(\kappa-\varepsilon)\varphi Y]\\
&-g(Z,\co_{Y}\xi)[h\varphi X+(\kappa-\varepsilon)\varphi X]\\
&+\eta(Z)\{\varepsilon[(\co_{X}h\varphi)Y-(\co_{Y}h\varphi)X]+(\varepsilon\kappa-1)[(\co_{X}\varphi)Y-(\co_{Y}\varphi)X]\}\\
&-\varepsilon\mu\{\eta(Y)(\co_{X}\varphi h)Z-\eta(X)(\co_{Y}\varphi h)Z+2g(X,\varphi Y)\varphi hZ\}.
\end{split}
\end{align}
Using now (\ref{031}), (\ref{052}) and Lemma \ref{042}, we have
\begin{align*}
(\co_{X}\varphi h)Y=&\{g(X,hY)+(\kappa-\varepsilon)g(X,-Y+\eta(Y)\xi)\}\xi\\
&+\eta(Y)[\varepsilon hX+(\varepsilon\kappa-1)(-X+\eta(X)\xi)]+\varepsilon\mu\eta(X)hY.
\end{align*}
Therefore, equation (\ref{044}), by using (\ref{031}) again, is reduced to (\ref{045}) and the proof is completed.
\end{proof}
\begin{theorem}\label{049}
Let $M^{2n+1}(\varphi,\xi,\eta,g)$be a $(\kappa, \mu)$-contact pseudo-metric manifold. If $\varepsilon\kappa< 1$, then for all $X_{\lambda},Z_{\lambda},Y_{\lambda}\in \di(\lambda)$ and $X_{-\lambda},Z_{-\lambda},Y_{-\lambda}\in \di(-\lambda)$, we have
\begin{align}
R(X_{\lambda},Y_{\lambda})Z_{-\lambda}&=(\kappa-\varepsilon\mu)[g(\varphi Y_{\lambda},Z_{-\lambda})\varphi X_{\lambda}-g(\varphi X_{\lambda},Z_{-\lambda})\varphi Y_{\lambda}],\label{41}\\
R(X_{-\lambda},Y_{-\lambda})Z_{\lambda}&=(\kappa-\varepsilon\mu)[g(\varphi Y_{-\lambda},Z_{\lambda})\varphi X_{-\lambda}-g(\varphi X_{-\lambda},Z_{\lambda})\varphi Y_{-\lambda}],\label{42}\\
R(X_{-\lambda},Y_{\lambda})Z_{-\lambda}&=-\kappa g(\varphi Y_{\lambda},Z_{-\lambda}) \varphi X_{-\lambda}-\varepsilon\mu g(\varphi Y_{\lambda},X_{-\lambda})\varphi Z_{-\lambda} ,\label{43}\\
R(X_{-\lambda},Y_{\lambda})Z_{\lambda}&=\kappa g(\varphi X_{-\lambda},Z_{\lambda}) \varphi Y_{\lambda}+\varepsilon\mu g(\varphi X_{-\lambda},Y_{\lambda})\varphi Z_{\lambda} ,\label{44}\\
R(X_{\lambda},Y_{\lambda})Z_{\lambda}&=[2(\varepsilon+\lambda)-\varepsilon\mu][g(Y_{\lambda},Z_{\lambda})X_{\lambda}-g(X_{\lambda},Z_{\lambda})Y_{\lambda}],\label{45}\\
R(X_{-\lambda},Y_{-\lambda})Z_{-\lambda}&=[2(\varepsilon-\lambda)-\varepsilon\mu][g(Y_{-\lambda},Z_{-\lambda})X_{-\lambda}-g(X_{-\lambda},Z_{-\lambda})Y_{-\lambda}],\label{46}
\end{align}
\end{theorem}
\begin{proof}
The first part of the Theorem follows from (\ref{030}) and Lemma \ref{037}.\\
Let $\{E_{1},\cdots,E_{n},E_{n+1}=\varphi E_{1},\cdots,E_{2n}=\varphi E_{n},E_{2n+1}=\xi\}$ be a (local) $\varphi$-basis of vector fields on $M$ with $E_{i}\in\di(\lambda)$ and so $\varphi E_{i}\in\di(-\lambda)$. For any index $i =1,\cdots,2n$,$\{\xi, E_{i}\}$ spans a non-degenerate plane on the tangent space at each point, where the basis is defined. Then, we have
\begin{align}\label{048}\begin{split}
R(X_{\lambda},Y_{\lambda})Z_{-\lambda}=&\sum_{i=1}^{n}\varepsilon_{i}\{g(R(X_{\lambda},Y_{\lambda})Z_{-\lambda},E_{i})E_{i}+g(R(X_{\lambda},Y_{\lambda})Z_{-\lambda},\varphi E_{i})\varphi E_{i}\}\\
&+\varepsilon g(R(X_{\lambda},Y_{\lambda})Z_{-\lambda},\xi)\xi.
\end{split}
\end{align}
But since $\xi$ belonging to the $(\kappa ,\mu)$-nullity distribution, using (\ref{62}), we easily have
$$g(R(X_{\lambda},Y_{\lambda})Z_{-\lambda},\xi)=-g(R(X_{\lambda},Y_{\lambda})\xi,Z_{-\lambda})=0.$$
By Proposition \ref{47}, we get
$$g(R(X_{\lambda},Y_{\lambda})Z_{-\lambda},E_{i})=-g(R(X_{\lambda},Y_{\lambda})E_{i},Z_{-\lambda})=0.$$
On the other hand, if $X\in\di(\lambda)$ and $Y,Z\in\di(-\lambda)$, then applying (\ref{045}), we get
$$hR(X,Y)Z+\lambda R(X,Y)Z=-2\lambda\{\kappa g(X,\varphi Z)\varphi Y+\varepsilon\mu g(X,\varphi Y)\varphi Z\},$$
and taking the inner product with $W\in\di(\lambda)$, we obtain
\begin{align}\label{047}
g(R(X,Y)Z,W)=-\kappa g(X,\varphi Z)g(\varphi Y,W)-\varepsilon\mu g(X,\varphi Y)g(\varphi Z,W),
\end{align}
for any $X,W\in\di(\lambda)$ and $Y,Z\in\di(-\lambda)$. Using (\ref{047}) and the first Bianchi identity, we calculate
\begin{align*}
\sum_{i=1}^{n}&\varepsilon_{i}g(R(X_{\lambda},Y_{\lambda})Z_{-\lambda},\varphi E_{i})\varphi E_{i}\\
=&-\sum_{i=1}^{n}\varepsilon_{i}g(R(Y_{\lambda},Z_{-\lambda})X_{\lambda},\varphi E_{i})\varphi E_{i}-\sum_{i=1}^{n}\varepsilon_{i}g(R(Z_{-\lambda},X_{\lambda})Y_{\lambda},\varphi E_{i})\varphi E_{i}\\
=&\sum_{i=1}^{n}\varepsilon_{i}g(R(Y_{\lambda},Z_{-\lambda})\varphi E_{i},X_{\lambda})\varphi E_{i}-\sum_{i=1}^{n}\varepsilon_{i}g(R(X_{\lambda},Z_{-\lambda})\varphi E_{i},Y_{\lambda})\varphi E_{i}\\
=&\sum_{i=1}^{n}\varepsilon_{i}\{-\kappa g(Y_{\lambda},\varphi^{2}E_{i})g(\varphi Z_{-\lambda},X_{\lambda})\varphi E_{i}-\varepsilon\mu g(Y_{\lambda},\varphi Z_{-\lambda})g(\varphi^{2}E_{i},X_{\lambda})\varphi E_{i}\}\\
&-\sum_{i=1}^{n}\varepsilon_{i}\{-\kappa g(X_{\lambda},\varphi^{2}E_{i})g(\varphi Z_{-\lambda},Y_{\lambda})\varphi E_{i}-\varepsilon\mu g(X_{\lambda},\varphi Z_{-\lambda})g(\varphi^{2}E_{i},Y_{\lambda})\varphi E_{i}\}\\
=&\kappa g(\varphi Z_{-\lambda},X_{\lambda})\varphi\sum_{i=1}^{n}\varepsilon_{i}g(Y_{\lambda},E_{i})E_{i}+\varepsilon\mu g(Y_{\lambda},\varphi Z_{-\lambda})\varphi\sum_{i=1}^{n}\varepsilon_{i}g(E_{i},X_{\lambda})E_{i}\\
&+\kappa g(Z_{-\lambda},\varphi Y_{\lambda})\varphi\sum_{i=1}^{n}\varepsilon_{i}g(X_{\lambda},E_{i})E_{i}+\varepsilon\mu g(\varphi X_{\lambda},Z_{-\lambda})\varphi\sum_{i=1}^{n}\varepsilon_{i}g(E_{i},Y_{\lambda})E_{i}\\
=&\kappa\{g(Z_{-\lambda},\varphi Y_{\lambda})\varphi X_{\lambda}-g(Z_{-\lambda},\varphi X_{\lambda})\varphi Y_{\lambda}\}\\
&+\varepsilon\mu\{g(\varphi X_{\lambda},Z_{-\lambda})\varphi Y_{\lambda}-g(\varphi Y_{\lambda},Z_{-\lambda})\varphi X_{\lambda}\}\\
=&(\kappa-\varepsilon\mu)\{g(Z_{-\lambda},\varphi Y_{\lambda})\varphi X_{\lambda}-g(Z_{-\lambda},\varphi X_{\lambda})\varphi Y_{\lambda}\}.
\end{align*}
Therefore, (\ref{048}) gives
$$R(X_{\lambda},Y_{\lambda})Z_{-\lambda}=(\kappa-\varepsilon\mu)\{g(Z_{-\lambda},\varphi Y_{\lambda})\varphi X_{\lambda}-g(Z_{-\lambda},\varphi X_{\lambda})\varphi Y_{\lambda}\}.$$
The proof of the remaining cases are similar and will be omitted.
\end{proof}
Then they showed the following.
\begin{theorem}
Let $M^{2n+1}(\varphi,\xi,\eta,g)$be a $(\kappa, \mu)$-contact pseudo-metric manifold. If $\varepsilon\kappa<1$, then for any $X$ orthogonal to $\xi$\\
$(i)$ the $\xi$-sectional curvature $K(X,\xi)$ is given by
\begin{equation*}
K(X,\xi)=\kappa+\mu\dfrac{g(hX,X)}{g(X,X)} =
\left\{
\begin{array}{lr}
\kappa+\lambda\mu,\quad\text{if $X\in\di(\lambda)$}, \\
\\
\kappa-\lambda\mu,\quad\text{if $X\in\di(-\lambda)$},
\end{array}\right.
\end{equation*}
$(ii)$ the sectional curvature of a plane section $(X,Y)$ normal to $\xi$ is given by
\begin{align}\label{025}
K(X,Y)=
\left\{
\begin{array}{lr}
2(\varepsilon+\lambda)-\varepsilon\mu,\quad\quad\quad\quad\text{ for any $X,Y\in\di(\lambda), n>1$},\\
\\
-(\kappa+\varepsilon\mu)\dfrac{g(X,\varphi Y)^{2}}{g(X,X)g(Y,Y)},\quad\text{ for any unit vectors $X\in\di(\lambda), Y\in\di(-\lambda)$},\\
\\
2(\varepsilon-\lambda)-\varepsilon\mu,\quad\quad\quad\quad\text{ for any $X,Y\in\di(-\lambda), n>1$},
\end{array}\right.
\end{align}
$(iii)$\quad The Ricci operator is given by
\begin{align}
\begin{split}\label{024}
QX=&\varepsilon[2(n-1)-n\mu]X+(2(n-1)+\mu)hX+[2(1-n)\varepsilon +2n\kappa +n\varepsilon\mu]\eta(X)\xi.
\end{split}
\end{align}
\end{theorem}
\begin{proof}
$(i)$ From (\ref{060}), if we set $Y=\xi$ in the relation of (\ref{62}), for $X$ orthogonal to $\xi$ from which, taking the inner product with $X$, we get
$$K(X,\xi)=\dfrac{\varepsilon\{\kappa g(X,X)+\mu g(hX,X)\}}{\varepsilon g(X,X)}.$$
So, we have
\begin{align*}
K(X,\xi)&=\kappa+\mu \dfrac{g(hX,X)}{g(X,X)}\\
&=\kappa+\mu \dfrac{\lambda g(hX_{\lambda},X_{\lambda})-\lambda g(hX_{-\lambda},X_{-\lambda})}{g(X_{\lambda},X_{\lambda})+g(X_{-\lambda},X_{-\lambda})},
\end{align*}
which is the required result.
\begin{itemize}
  \item[(ii)] This follows immediately from Theorem \ref{049}.
  \item[(iii)] The first consider a $\varphi$-basis $\{E_{1}, \ldots, E_{n}, E_{n+1}=\varphi E_{1}, \ldots,$ $ E_{2n}=\varphi E_{n}, E_{2n+1}=\xi \}$ of vector fields on $M$.
\end{itemize}
For any index $i =1,\ldots, 2n$,$\{\xi, E_{i}\}$ spans a non-degenerate plane on the tangent space at each point, where the basis is defined. Putting $Y=Z=E_{i}$ in $R(X,Y)Z$, adding with respect to index of $i$ and using (\ref{001}), (\ref{002}) and (\ref{013}), we get the following formula, for the Ricci operator, at any point of $M$:
$$QX=\sum_{i=1}^{n}\varepsilon_{i}\{R(X,E_{i})E_{i}+R(X,\varphi E_{i})\varphi E_{i}\}+\varepsilon R(X,\xi)\xi.$$
Suppose now that $X$ is arbitrary vector fields and write
$$X=X_{\lambda}+X_{-\lambda}+\eta(X)\xi,$$
On the other hand, from Theorem \ref{049}, we have
\begin{align*}
QX=&\sum_{i=1}^{n}\varepsilon_{i}\{R(X_{\lambda},E_{i})E_{i}+R(X_{-\lambda},E_{i})E_{i}+\eta(X)R(\xi,E_{i})E_{i}+R(X_{\lambda},\varphi E_{i})\varphi E_{i}\\
&+\eta(X)R(\xi,\varphi E_{i})\varphi E_{i}+R(X_{-\lambda},\varphi E_{i})\varphi E_{i}\}+\varepsilon R(X_{\lambda},\xi)\xi+\varepsilon R(X_{-\lambda},\xi)\xi\\
=&[2(\varepsilon+\lambda)-\varepsilon\mu](n-1)X_{\lambda}-(\kappa+\varepsilon\mu)X_{-\lambda}+n\kappa\eta(X)\xi-(\kappa+\varepsilon\mu)X_{\lambda}\\
&+[2(\varepsilon-\lambda)-\varepsilon\mu](n-1)X_{-\lambda}+(\kappa+\varepsilon\mu)X_{\lambda}+n\kappa\eta(X)\xi+(\kappa+\mu h)X_{-\lambda}\\
=&\varepsilon[(2-\mu)(n-1)-\mu](X_{\lambda}+X_{-\lambda})\\
&+[2(n-1)+\mu]h(X_{\lambda}+X_{-\lambda})+2n\kappa\eta(X)\xi.
\end{align*}
So, the relation of (\ref{024}) is obtained.
\end{proof}

\begin{theorem}\label{087}
 Let $M^{2n+1}(\eta, \xi, \varphi, g)$ be a $(\kappa, \mu)$-contact pseudo-metric manifold and $n>1$. If the $\varphi$-sectional curvature of any point of $M$ is independent of the choice of $\varphi$-section at the point, then it is constant on $M$ and the curvature tensor
is given by
\begin{align}\label{022}
\begin{split}
R(X, Y)Z=&(\frac{c+3\varepsilon}{4})\{g(Y, Z)X-g(X, Z)Y\}\\
&+(\dfrac{c-\varepsilon}{4})\{2g(X, \varphi Y)\varphi Z+g(X, \varphi Z)\varphi Y-g(Y, \varphi Z)\varphi X\}\\
&+(\dfrac{c+3\varepsilon}{4}-\kappa)\{\varepsilon\eta(X)\eta(Z)Y-\varepsilon\eta(Y)\eta(Z)X+\eta(Y)g(X, Z)\xi-\eta(X)g(Y, Z)\xi\}\\
&+\{-g(X, Z)hY-g(hX, Z)Y+g(Y, Z)hX+g(hY, Z)X\}\\
&+\dfrac{\varepsilon}{2}\{-g(hX, Z)hY+g(hY, Z)hX+g(\varphi hX, Z)\varphi hY-g(\varphi hY, Z)\varphi hX\}\\
&+(1-\mu)\{\varepsilon\eta(X)\eta(Z)hY+\eta(Y)g(hX, Z)\xi-\varepsilon\eta(Y)\eta(Z)hX-\eta(X)g(hY, Z)\xi\},
\end{split}
\end{align}
where $c$ is the constant $\varphi$-sectional curvature. Moreover if $\kappa\neq\varepsilon$, then $\mu=\varepsilon\kappa+1$ and
$c=-2\kappa-\varepsilon$.
\end{theorem}
\begin{proof}
 For the Sasakian case $\kappa=\varepsilon$, the proof is known (\cite{Takahashi:SasakianManifoldWithPseud}). So, we have to prove the theorem for $\kappa\neq \varepsilon$. Let $p\in M$ and $X,$ $Y\in T_{p}M$ orthogonal to $\xi$. Using the first identity of Bianchi, the basic properties of the curvature tensor, $\varphi$ is antisymmetric, $h$ is symmetric, (\ref{001}) and (\ref{002}), we obtain from (\ref{006}), successively:
\begin{align}\label{007}\begin{split}
g(R(X,\varphi X)Y,\varphi Y)=&g(R(X,\varphi Y)Y,\varphi X)+g(R(X,Y)X,Y)-\varepsilon g(X, Y)^{2}-\varepsilon g(hX, Y)^{2} \\
&-2g(X, Y)g(hX, Y)+\varepsilon g(X, X)g(Y, Y)+g(X, X)g(hY, Y) \\
&+g(Y, Y)g(hX, X)+\varepsilon g(hX, X)g(hY, Y)-\varepsilon g(\varphi X, Y)^{2} \\
&+\varepsilon g(\varphi hX, Y)^{2}-\varepsilon g(\varphi hX, X)g(\varphi hY,Y),
\end{split}
\end{align}
\begin{align}\label{008}\begin{split}
g(R(X, \varphi Y)X,\varphi Y)=&g(R(X, \varphi Y)Y,\varphi X)+\varepsilon g(X, Y)^{2}-\varepsilon g(hX, Y)^{2} \\
&-\varepsilon g(\varphi hX, X)g(\varphi hY, Y)-\varepsilon g(X, X)g(Y, Y)-g(Y, Y)g(hX, X) \\
&+g(X, X)g(hY, Y)+\varepsilon g(hX, X)g(hY, Y)+\varepsilon g(\varphi X, Y)^{2} \\
&+\varepsilon g(\varphi hX, Y)^{2}+2g(\varphi X, Y)g(\varphi hX,Y),\end{split}
\end{align}
\begin{align}\label{009}\begin{split}
g(R(Y, \varphi X)Y,\varphi X)=&g(R(X, \varphi Y)Y,\varphi X)+\varepsilon g(X, Y)^{2}-\varepsilon g(hX, Y)^{2} \\
&-\varepsilon g(\varphi hX, X)g(\varphi hY, Y)+\varepsilon g(\varphi X, Y)^{2}+\varepsilon g(\varphi hX, Y)^{2} \\
&-2g(\varphi X, Y)g(\varphi hX, Y)-\varepsilon g(X, X)g(Y, Y)-g(X, X)g(hY, Y) \\
&+g(Y, Y)g(hX, X)+\varepsilon g(hX, X)g(hY, Y),\end{split}
\end{align}
\begin{align}\label{010}\begin{split}
g(R(X, Y)\varphi X,\varphi Y)=&g(R(X, Y)X,Y)-\varepsilon g(X, Y)^{2}-\varepsilon g(hX, Y)^{2} \\
&-2g(X, Y)g(hX, Y)+\varepsilon g(X, X)g(Y, Y)+g(X, X)g(hY, Y) \\
&+g(Y, Y)g(hX, X)+\varepsilon g(hX, X)g(hY, Y)-\varepsilon g(\varphi X, Y)^{2} \\
&+\varepsilon g(\varphi hX, Y)^{2}-\varepsilon g(\varphi hX, X)g(\varphi hY, Y).\end{split}
\end{align}
We now suppose that the $\varphi$-sectional curvature at $p$ is independent of the $\varphi$-section at $p$, i.e. $K(X, \varphi X)=c(p)$ for any $X\in T_{p}M$ orthogonal to $\xi$. Let $X,Y\in T_{p}M$ and $X,Y$
orthogonal to $\xi$ . From
\begin{align*}
g(R(X+Y, \varphi X+\varphi Y)(X+Y),\varphi X+\varphi Y)&=-c(p) g(X+Y, X+Y)^{2},\\
g(R(X-Y, \varphi X-\varphi Y)(X-Y),\varphi X-\varphi Y)&=-c(p)g(X-Y, X-Y)^{2},
\end{align*}
we get by a straightforward calculation
\begin{align}\label{011}\begin{split}
2g(R(X,\varphi X)Y,\varphi Y)&+g(R(X, \varphi Y)X,\varphi Y)+2g(R(X, \varphi Y)Y,\varphi X)+g(R(Y, \varphi X)Y,\varphi X) \\
&=-2c(p)\{2g(X, Y)^{2}+g(X, X)g(Y, Y)\}.\end{split}
\end{align}
Thus with combining (\ref{007}), (\ref{008}), (\ref{009}) and (\ref{011}), we get
\begin{align}\label{012}\begin{split}
&3g(R(X, \varphi Y)Y,\varphi X)+g(R(X, Y)X,Y)-2\varepsilon g(hX, Y)^{2} \\
&-2g(X, Y)g(hX, Y)+g(X, X)g(hY, Y)+g(Y, Y)g(hX, X) \\
&+2\varepsilon g(hX, X)g(hYY)+2\varepsilon g(\varphi hX, Y)^{2}-2\varepsilon g(\varphi hX, X)g(\varphi hY, Y) \\
&=-c(p)\{2g(X, Y)^{2}+g(X, X)g(Y, Y)\}.\end{split}
\end{align}
Now, we replace $Y$ by $\varphi Y$ in (\ref{012}), then using (\ref{009}) and (\ref{013}), we have
\begin{align}\label{014}\begin{split}
&-3g(R(X, Y)\varphi Y,\varphi X)+g(R(X, \varphi Y)X,\varphi Y)-2g(\varphi hX, Y)^{2} \\
&+2g(X, \varphi Y)g(\varphi hX, Y)-g(X, X)g(hY, Y)+g(Y, Y)g(hX, X) \\
&-2g(hX, X)g(hY, Y)+2g(hX, Y)^{2}+2g(\varphi hX, X)g(\varphi hY, Y) \\
&=-c(p)\{2g(X, \varphi Y)^{2}+g(X, X)g(Y, Y)\}.\end{split}
\end{align}
On the other hand, with combining  the relation of (\ref{014}) with (\ref{008}) and (\ref{010}), we get
\begin{align}\label{015}\begin{split}
&3g(R(X, Y)X,Y)+g(R(X, \varphi Y)Y,\varphi X)-2\varepsilon g(X, Y)^{2} \\
&-2\varepsilon g(hX, Y)^{2}-6g(X, Y)g(hX, Y)+2\varepsilon g(X, X)g(Y, Y) \\
&+3g(X, X)g(hY, Y)+3g(Y, Y)g(hX, X)+2\varepsilon g(hX, X)g(hY, Y) \\
&-2\varepsilon g(X, \varphi Y)^{2}+2\varepsilon g(\varphi hX, Y)^{2}-2\varepsilon g(\varphi hX,X)g(\varphi hY, Y) \\
&=-c(p)\{2g(X, \varphi Y)^{2}+g(X, X)g(Y, Y)\}.\end{split}
\end{align}
Now for any $X,Y\in T_{p}M$ and $X,Y$ orthogonal to $\xi$, (\ref{015}) together with (\ref{012}) yield
\begin{align}\label{016}\begin{split}
4g(R(X,Y)Y,X)=&(c(p)+3\varepsilon)\{g(X,X)g(Y,Y)-g(X,Y)^{2}\}+3(c(p)-\varepsilon)g(X,\varphi Y)^{2} \\
&-2\{\varepsilon g(hX,Y)^{2}+4g(X,Y)g(hX,Y)-2g(X,X)g(hY,Y)-2g(Y,Y)g(hX,X) \\
&-\varepsilon g(hX,X)g(hY,Y)-\varepsilon g(\varphi hX,Y)^{2}+\varepsilon g(\varphi hX,X)g(\varphi hY,Y)\}.\end{split}
\end{align}
Let $X,Y,Z\in T_{p}M$ and $X,Y,Z$ orthogonal to $\xi$. Applying (\ref{016}) in
$$g(R(X+Z, Y)Y,X+Z)=g(R(X, Y)Y,X)+g(R(Z, Y)Y,Z)+2g(R(X, Y)Y,Z).$$
Finally, we get
\begin{align}\label{017}\begin{split}
4g(R(X, Y)Y,Z)=&(c(p)+3\varepsilon)\{g(X, Z)g(Y, Y)-g(X,Y)g(Y, Z)\} \\
&+3(c(p)-\varepsilon)g(X, \varphi Y)g(Z, \varphi Y)-2\{\varepsilon g(hX, Y)g(hZ,Y)+2g(X,Y)g(hZ, Y) \\
&+2g(Z, Y)g(hX, Y)-2g(X, Z)g(hY, Y)-2g(Y, Y)g(hX, Z) \\
&-\varepsilon g(hX, Z)g(hY, Y)-\varepsilon g(\varphi hX, Y)g(\varphi hZ, Y)+\varepsilon g(\varphi hX, Z)g(\varphi hY, Y)\}.\end{split}
\end{align}
Moreover, using (\ref{013}), (\ref{62}) and $ h\varphi$ is symmetric, it is easy to check that (\ref{017}) is valid for any $Z$ and for $X,Y$ orthogonal to $\xi$. Hence for any $X,Y$ orthogonal to $\xi$, the relation of (\ref{017}) is reduced to
\begin{align}\label{018}\begin{split}
4R(X, Y)Y=&(c(p)+3\varepsilon)\{g(Y, Y)X-g(X, Y)Y\}+3(c(p)-\varepsilon)g(X, \varphi Y)\varphi Y \\
&-2\{\varepsilon g(hX, Y)hY+2g(X, Y)hY+2g(hX, Y)Y-2g(hY, Y)X \\
&-2g(Y, Y)hX-\varepsilon g(hY, Y)hX-\varepsilon g(\varphi hX, Y)\varphi hY+\varepsilon g(\varphi hY, Y)\varphi hX\}.\end{split}
\end{align}
Now, let $X,Y,Z$ be orthogonal to $\xi$. Replacing $Y$ by $Y+Z$ in (\ref{018}). Then from
$$R(X,Y+Z)(Y+Z)=R(X,Y)Y+R(X,Z)Z+R(X,Y)Z+R(X,Z)Y,$$
we get
\begin{align}\label{019}\begin{split}
4\{R(X, Y)Z+R(X, Z)Y\}=&(c(p)+3\varepsilon)\{2g(Y, Z)X-g(X, Y)Z-g(X, Z)Y\} \\
&+3(c(p)-\varepsilon)\{g(X, \varphi Y)\varphi Z+g(X, \varphi Z)\varphi Y\} \\
&-2\{\varepsilon g(hX, Y)hZ+g(hX, Z)hY+2g(X, Y)hZ \\
&+2g(X, Z)hY+2g(hX, Y)Z+2g(hX, Z)Y-4g(hY, Z)X \\
&-4g(Y, Z)hX-2\varepsilon g(hY, Z)hX-\varepsilon g(\varphi hX, Y)\varphi hZ \\
&-\varepsilon g(\varphi hX, Z)\varphi hY+2\varepsilon g(\varphi hY, Z)\varphi hX\}.\end{split}
\end{align}
Replacing $X$ by $Y$ and $Y$ by $-X$ in (\ref{019}), we have
\begin{align}\label{020}\begin{split}
4\{R(X, Y)Z+R(Z, Y)X\}=&(c(p)+3\varepsilon)\{-2g(X, Z)Y+g(X, Y)Z+g(Y, Z)X\} \\
&+3(c(p)-\varepsilon)\{-g(\varphi X, Y)\varphi Z-g(\varphi Z, Y)\varphi X\} \\
&-2\{-\varepsilon g(hY, X)hZ-\varepsilon g(hY, Z)hX-2g(X, Y)hZ \\
&-2g(Y, Z)hX-2g(X, hY)Z-2g(hY, Z)X \\
&+4g(hX, Z)Y+4g(X, Z)hY+2\varepsilon g(hX, Z)hY \\
&+\varepsilon g(\varphi hY, X)\varphi hZ+\varepsilon g(\varphi hY, Z)\varphi hX-2\varepsilon g(\varphi hX, Z)\varphi hY\}.\end{split}
\end{align}
Adding (\ref{019}) and (\ref{020}) and using Bianchi first identity, $\varphi$ is antisymmetric and $\varphi h$ is symmetric, we get
\begin{align}\label{021}\begin{split}
4R(X, Y)Z=&(c(p)+3\varepsilon)\{g(Y, Z)X-g(X, Z)Y\} \\
&+(c(p)-\varepsilon)\{2g(X, \varphi Y)\varphi Z+g(X, \varphi Z)\varphi Y-g(Y, \varphi Z)\varphi X\} \\
&-2\{\varepsilon g(hX, Z)hY+2g(X, Z)hY+2g(hX, Z)Y-2g(hY, Z)X \\
&-2g(Y, Z)hX-\varepsilon g(hY, Z)hX-\varepsilon g(\varphi hX, Z)\varphi hY+\varepsilon g(\varphi hY, Z)\varphi hX\},\end{split}
\end{align}
for any $X,Y,Z$ orthogonal to $\xi$. Moreover, using (\ref{62}), (\ref{052}) and the first part of (\ref{001}), we conclude that (\ref{021}) is valid for any $Z$ and for $X,Y$ orthogonal to $\xi$. Now, let $X,Y,Z$ be arbitrary vectors of $T_{p}M$. Writing
$$X=X_{T}+\eta(X)\xi,\quad Y=Y_{T}+\eta(Y)\xi,$$
where $g(X_{T}, \xi)=g(Y_{T}, \xi)=0$ , and using (\ref{62}), (\ref{023}) and (\ref{052}), then (\ref{021}) gives (\ref{022}) after a straightforward calculation.\\
Now, we will prove that the $\varphi$-sectional curvature is constant. Consider a (local) $\varphi$-basis $\{E_{1},\cdots, E_{n},E_{n+1}=\varphi E_{1}, \cdots,E_{2n}=\varphi E_{n},E_{2n+1}=\xi\}$ of vector fields on $M$. For any index $i =
1,\cdots,2n$,$\{\xi, E_{i}\}$ spans a non-degenerate plane on the tangent space at each point, where the basis is defined. Putting $Y=Z=E_{i}$ in (\ref{022}), adding with respect to $i$ and using (\ref{001}), (\ref{002}) and (\ref{013}), we get the following formula, for the Ricci operator, at any point of $M$:
\begin{align*}\begin{split}
2Q=&\{(n+1)c+3\varepsilon(n-1)+2\kappa\}I \\
&-\{(n+1)c+3\varepsilon(n-1)-2\kappa(2n-1)\}\eta\otimes\xi \\
&+2\{2(n-1)+\mu\}h.\end{split}
\end{align*}
Comparing this with (\ref{024}), which is valid on any $(\kappa, \mu)$-contact pseudo-metric manifold  with $\kappa\neq\varepsilon$, we get
\begin{align}\label{061}
(n+1)c=(n-1)\varepsilon-2n\varepsilon\mu-2\kappa,
\end{align}
i.e. $c$ is constant. On the other hand, from (\ref{025}), we have
\begin{align}\label{062}
 c=-(\kappa+\varepsilon\mu).
\end{align}
Comparing (\ref{061}) and (\ref{062}), we get $(n-1)(\varepsilon\mu-\kappa-\varepsilon)=0$. Moreover, since $n>1$, we have $\mu=\varepsilon\kappa+1$ and so $c=-2\kappa-\varepsilon$. This completes the proof of the theorem.
\end{proof}
\begin{theorem}\label{088}
Let $M^{2n+1}(\varphi,\xi,\eta,g)$ be a $(\kappa, \mu)$-contact pseudo-metric manifold with $\varepsilon\kappa<1$ and $n>1$. Then $M$ has constant $\varphi$-sectional curvature if and
only if $\mu=\varepsilon\kappa+1$ .
\end{theorem}
\begin{proof}
In Theorem \ref{087}, we proved that $\mu=\varepsilon\kappa+1$, in the case where the non Sasakian, $(\kappa, \mu)$-contact pseudo-metric manifold has constant $\varphi$-sectional curvature. Now, we will prove the inverse, i.e. supposing $M(\varphi,\xi,\eta,g)$ is a $(2n+1)$-dimensional $(n>1)$, non Sasakian, $(\kappa, \mu)$-contact pseudo-metric manifold with
\begin{align}\label{a}
\mu=\varepsilon\kappa+1.
\end{align}
We will prove that $M$ has constant $\varphi$-sectional curvature. Let $X\in T_{p}M$ be a unit vector orthogonal to $\xi$. By Theorem \ref{037}, we can write
\begin{center}
$X=X_{\lambda}+X_{-\lambda}\qquad\text{ where}\quad X_{\lambda}\in \di(\lambda) \quad\text{and}\quad X_{-\lambda}\in \di(-\lambda)$.
\end{center}
Using Theorem \ref{037}, Theorem \ref{049} and a long straightforward calculation, we get
$$K(X, \varphi X)=-(\kappa+\varepsilon\mu)+4(\kappa-\varepsilon\mu+\varepsilon)(g(X_{\lambda}, X_{\lambda})g(X_{-\lambda}, X_{-\lambda})-g(X_{\lambda}, \varphi X_{-\lambda})^{2})$$
and hence by (\ref{a}), we have $K(X, \varphi X)=-(\kappa+\varepsilon\mu)=\text{const}$.
\end{proof}
\begin{example}[\cite{GhaffarzadehFaghfouri2}, Theorem 4.1]
 The tangent sphere bundle $T_{\varepsilon}M$ is $(\kappa, \mu)$-contact pseudo-metric manifold
 if and only if the base manifold $M$ is of constant sectional curvature $\varepsilon$ and $\kappa=3\varepsilon-2, \mu=-2\varepsilon$.
 \end{example}
 \begin{theorem}
Let $M$ be an $n$-dimensional pseudo-metric manifold, $n>2$, of constant sectional curvature $c$. The tangent sphere bundle $T_{\varepsilon}M$ has constant $\varphi$-sectional curvature $(-4c(\varepsilon-1)+ c^{2})$ if and only if $c=2\varepsilon\pm\sqrt{4+\varepsilon}$.
\end{theorem}
\begin{example}
Let $M(\varphi,\xi,\eta,g)$ be a contact pseudo-metric manifold of dimension $2n+1$, with $g(\xi,\xi)=\varepsilon$.
Then, it is easy to check that, for any real constant $a>0$ and by choosing the tensors
$$\overline{\eta}=a\eta,\quad\overline\xi=\frac{1}{a}\xi,\quad\overline{\varphi}=\varphi,\quad\overline{g}=ag+\varepsilon a(a-1)\eta\otimes\eta,$$
 $M(\overline{\varphi},\overline{\xi},\overline{\eta},\overline{g})$ is a new $(\overline{\kappa},\overline{\mu})$-contact pseudo-metric manifold with
\begin{align}\label{089}
\overline{\kappa}=\frac{\kappa+\varepsilon a^{2}-\varepsilon}{a^{2}}, \quad\quad \overline{\mu}=\frac{\mu+2a-2}{a}.
\end{align}
Now, we find the value of $a$, so that $M$ has constant $\varphi$-sectional curvature. Using  Theorem \ref{088}, we must  have $\overline{\mu}=\varepsilon\overline{\kappa}+1$. So, we get
$a=\dfrac{(\varepsilon\kappa-1)}{(\mu-2)}$.
 In fact with choosing $a=\dfrac{(\varepsilon\kappa-1)}{(\mu-2)}>0$, $M(\overline{\varphi},\overline{\xi},\overline{\eta},\overline{g})$ has constant $\varphi$-sectional curvature
$\overline{c}=-\overline{\kappa}-\varepsilon\overline{\mu}=\varepsilon(1-2\overline{\mu})=\dfrac{(2(\mu-2)^{2}-3(1-\varepsilon\kappa))}{(\varepsilon-\kappa)}$.
\end{example}
\section{generalized  $(\kappa, \mu)$-contact pseudo-metric manifold}
 In \eqref{61} and \eqref{62}, if $\kappa$ and $\mu$ are real smooth functions on $M$, we call generalized  $(\kappa, \mu)$-contact pseudo-metric manifold.

\begin{example}
We consider the $3$-dimensional manifold $M =\{(x,y,z) \in \mathbb{R}^3 | z\neq 0\}$, where $(x,y,z)$ are the standard coordinates in $\mathbb{R}^3$. The vector fields
$$e_{1}=\frac{\partial}{\partial x},\quad e_{2}=\frac{1}{z^{2}}\frac{\partial}{\partial y},\quad e_{3}=2yz^{2}\frac{\partial}{\partial x}+\frac{2x}{z^{6}}\frac{\partial}{\partial y}+\frac{1}{z^{6}}\frac{\partial}{\partial z}$$
are linearly independent at each point of $M$. Let $g$ be the pseudo-Riemannian metric defined by
$g(e_{i},e_{j}) = \varepsilon_{i}\delta_{ij},$ where $ i, j = 1, 2, 3$ and  $\varepsilon_{1}=\varepsilon,\quad\varepsilon_{2}=\varepsilon_{3}=1$. Let $\co$ be the Levi-Civita connection and $R$ the curvature tensor of $g$. We easily get
$$[e_{1},e_{2}]=0,\quad [e_{1},e_{3}]=\frac{2}{z^{4}}e_{2},\quad [e_{2},e_{3}]=2\left(e_{1}+\frac{1}{z^{7}}e_2\right)$$
Let $\eta$ be the $1$-form defined by $\eta(X) =\varepsilon g(X, e_{1})$ for $X\in\Gamma(TM)$. $\eta$ is a contact form. Let $\varphi$ be the $(1,1)$-tensor field, defined by
$\varphi e_{1} = 0,\quad\varphi e_{2} =e_{3},\quad\varphi e_{3} =-e_{2}$.
So, $(\varphi,\xi=e_{1},\eta,g)$ defines a contact pseudo-metric structure on $M$. Now using the Koszul formula, we calculate
\begin{align*}
&\co_{e_{1}}e_{2}=-(\varepsilon+\frac{1}{z^{4}})e_{3},\quad \co_{e_{2}}e_{1}=-(\varepsilon+\frac{1}{z^{4}})e_{3},\\
&\co_{e_{1}}e_{3}=(\varepsilon+\frac{1}{z^{4}})e_{2},\quad \co_{e_{3}}e_{1}=(\varepsilon-\frac{1}{z^{4}})e_{2},\\
&\co_{e_{2}}e_{3}=(1+\frac{\varepsilon}{z^{4}})e_{1}+\frac{2}{z^{7}}e_{2},\quad\co_{e_{3}}e_{2}=(\frac{\varepsilon}{z^{4}}-1)e_{1},\\
&\co_{e_{2}}e_{2}=-\frac{2}{z^{7}}e_{3},\quad \co_{e_{1}}e_{1}=\co_{e_{3}}e_{3}=0.
\end{align*}
Also Using \eqref{050}, we obtain $he_{1}=0$, $he_{2}=\lambda e_{2},he_{3}=-\lambda e_{3}$, where $\lambda=\dfrac{1}{z^{4}}$. Now, putting $\mu=2(1+\varepsilon\lambda)$ and $\kappa=\varepsilon(1-\lambda^{2})$, we finally get
\begin{align*}
R(e_{2},e_{1})e_{1}&=\varepsilon(\kappa+\lambda\mu)e_{2},\\
R(e_{3},e_{1})e_{1}&=\varepsilon(\kappa-\lambda\mu)e_{3},\\
R(e_{2},e_{3})e_{1}&=0.
\end{align*}
These relations yield the following, by direct calculations,
$$R(X,Y)\xi=\varepsilon\kappa\{\eta(Y)X-\eta(X)Y\}+\varepsilon\mu\{\eta(Y)hX-\eta(X)hY\},$$
where $\kappa$ and $\mu$ are non-constant smooth functions. Hence $M$ is a generalized $(\kappa,\mu)$-contact
pseudo-metric manifold.
\end{example}
\begin{theorem}\label{th:3.6}
On a non Sasakian, generalized  $(\kappa,\mu)$-contact pseudo-metric manifold $M^{2n+1}$ with $n>1$, the functions $\kappa, \mu$ are constant, i.e., $M^{2n+1}$ is a $(\kappa,\mu)$-contact pseudo-metric manifold.
\end{theorem}
\begin{theorem}
Let $M$ be a non Sasakian, generalized $(\kappa,\mu)$-contact pseudo-metric manifold. If $\kappa, \mu$ satisfy the condition $a\kappa+b\mu =c$, where $a,b$ and $c$ are constant. Then $\kappa, \mu$ are constant.
\end{theorem}
The proof of the two previous theorems is the same with Theorem 3.5 and Theorem 3.6 in
\cite{Koufogiorgos:OnTheExistenceContactMetricManifolds} for contact
metric case. Therefore, we omit them here.



\end{document}